\documentclass[a4paper,11pt,reqno]{article}

\usepackage{amsmath,amsthm,amsrefs,amssymb}

\usepackage{color}

\newtheorem{theorem}{Theorem}

\newtheorem{lemma}[theorem]{Lemma}
\newtheorem*{maintheorem}{Main Theorem}

\theoremstyle{definition}

\newtheorem{remark}[theorem]{Remark}

\makeatletter

\DeclareMathOperator{\sing}{sing}
\newcommand{\singZ}{\#}

\DeclareMathOperator{\CAT}{CAT}

\newcommand{\uu}[1]{\underline{\underline{#1}}}

\DeclareMathOperator{\pd}{pd}
\DeclareMathOperator{\cd}{cd}
\DeclareMathOperator{\gd}{gd}

\newcommand{\calO}{\mathcal{O}}

\newcommand{\uupd}{\mathop{\uu\pd}}
\newcommand{\uucd}{\mathop{\uu\cd}}
\newcommand{\uugd}{\mathop{\uu\gd}}

\newcommand{\ucd}{\mathop{\underline\cd}}
\newcommand{\ugd}{\mathop{\underline\gd}}

\newcommand{\frakF}{\mathfrak{F}}
\newcommand{\Ffin}{\frakF_{\text{fin}}}
\newcommand{\Fvc}{\frakF_{\text{vc}}}

\newcommand{\OC}[2]{\mathop{\calO_{#2}#1}\nolimits}

\newcommand{\OFG}{\OC{G}{\frakF}}

\DeclareMathOperator{\Ext}{Ext}

\newcommand{\ModOFG}{\mathop{{\operator@font Mod\text{-}}\calO_{\frakF}G}}
\newcommand{\OFGMod}{\mathop{\calO_{\frakF}G\text{-}{\operator@font Mod}}}

\newcommand{\Z}{\mathbb{Z}}

\newcommand{\OFW}{\mathcal{O}_{\frakF}W}

\newcommand{\homotop}{\simeq}
\newcommand{\isom}{\cong}

\renewcommand{\:}{\mathord{:}\hspace*{1ex plus .25ex minus .3ex}}

\makeatother

\begin{document}

\title{An Eilenberg--Ganea Phenomenon for
Actions with Virtually Cyclic Stabilisers}

\author{Martin G.~Fluch\\ 
Bielefeld University\\
PO Box 100131\\
33501 Bielefeld\\
Germany\\
\texttt{mfluch@math.uni-bielefeld.de}
\and 
Ian J.~Leary\\
School of Mathematics\\
University of Southampton\\
Southampton  SO17 1BJ\\
United Kingdom\\
\texttt{I.J.Leary@soton.ac.uk}}

\maketitle

\begin{abstract}
In dimension $3$ and above, Bredon cohomology gives an accurate purely
algebraic description of the minimal dimension of the classifying space for
actions of a group with stabilisers in any given family of subgroups. For some
Coxeter groups and the family of virtually cyclic subgroups we show that the
Bredon cohomological dimension is~$2$ while the Bredon geometric dimension
is~$3$.
\end{abstract}

\thispagestyle{empty}


\section{Introduction and Preliminaries}

For a discrete group $G$, a family of subgroups $\frakF$ is a non-empty
collection of subgroups of $G$ that is closed under conjugation and taking
subgroups.  If $\frakF$ is a family of subgroups of $G$ then a model for
$E_{\frakF}G$, the classifying space for $G$-actions with stabilisers in
$\frakF$, is a $G$-CW-complex $X$ such that for $H\leq G$, the fixed point
set~$X^{H}$ is empty if $H\notin\frakF$ and is contractible if~$H\in \frakF$.
For any $G$ and $\frakF$ there is always a model for $E_{\frakF}G$ and it is
unique up to equivariant homotopy.

In the case when $\frakF$ consists of just the trivial group,
$E_{\frakF}G$ is the same thing as~$EG$, the universal cover of an
Eilenberg--Mac~Lane space for~$G$.  In the case when~$\frakF$ is the 
family $\Ffin(G)$ of all finite subgroups of~$G$ (respectively the family
$\Fvc(G)$ of all virtually cyclic
subgroups of~$G$) we write $\underline EG$ (respectively $\uu EG$) for
$E_{\frakF}G$.  The minimal dimension of any model for $E_{\frakF}G$
is denoted by $\gd_{\frakF}G$ and is called the \emph{Bredon
  geometric dimension} of $G$.

Homological algebra over the group ring $\Z G$ can be used to study
models for~$EG$, and Bredon cohomology is the natural generalisation
for studying models for~$E_{\frakF}G$.  In Bredon cohomology the orbit
category $\OFG$ replaces the group~$G$.  The orbit category $\OFG$ is
the category with objects the $G$-sets $G/H$ with $H\in \frakF$ and
$G$ maps as morphisms.  A (right) $\OFG$-module is then a
contravariant functor from the orbit category $\OFG$ to the category
of abelian groups.  In the case when $\frakF$
consists of just the trivial group,~$\OFG$ is a category with one object and
morphism set~$G$ and~$\OFG$-modules are the same as $\Z G$-modules. 

The category of $\OFG$-modules is an abelian category with enough projectives. 
The \emph{Bredon cohomological dimension} $\cd_{\frakF}G$ is defined to be the
projective dimension of the trivial $\OFG$-module $\underline \Z$, which takes
the value~$\Z$ on any object of $\OFG$ and which maps any morphism to the
identity.  The derived functors of the morphism functor in the category of
Bredon modules over $\OFG$ are denoted by $\Ext^*_\frakF(-,-)$. The \emph{Bredon
cohomology groups} of $G$ with coefficients the $\OFG$-module $M$ are the
abelian groups~$H_\frakF^*(G; M) = \Ext_\frakF^*(\underline\Z; M)$. For details
on Bredon cohomology we refer to~\cite{luck-89} or~\cite{fluch-phdthesis}.

If the family $\frakF$ consists of the trivial subgroup only, then
$\gd_\frakF G$ is the minimal dimension $\gd G$ an Eilenberg--Mac~Lane space for
$G$ can have. If~$\frakF$ is the family~$\Ffin(G)$ 
(respectively $\Fvc(G)$) then we use the notation $\ugd G$ (respectively~$\uugd
G$) for~$\gd_\frakF G$.

\medskip

As in the classical case a model for $E_\frakF G$ gives rise to a resolution of
the trivial $\OFG$-module~$\underline\Z$ by projective $\OFG$-modules. Therefore
$\cd_\frakF G \leq \gd_\frakF G$ in general. If $\cd_\frakF G\geq 3$, then
$\cd_\frakF G=\gd_\frakF G$. In the classical case, that is when $\frakF=\{1\}$
consists only of the trivial subgroup, this is due to
Eilenberg--Ganea~\cite{eilenberg-57}. For $\frakF = \Ffin(G)$ this was proved
in~\cite{luck-89} and this proof generalises to arbitrary families~$\frakF$,
cf.~Theorem~0.1 in~\cite{luck-00}*{p.~294}. In the classical case it is well
known that $\cd_\frakF G=0$ implies $\gd_\frakF G=0$ and for general families
this implication follows from Lemma~2.5 in~\cite{symonds-05}*{p.~265}.

In the classical case, the statement that the cohomological and geometric
dimension always agree is known as the Eilenberg--Ganea Conjecture. Since the
work of Stallings~\cite{stallings-68a} and Swan~\cite{swan-69} implies that $\cd
G=1$ if and only if $\gd G=1$, this conjecture can only be falsified by a group
$G$ with~$\cd G=2$ but $\gd G=3$.

In \cite{brady-01} right-angled Coxeter groups $W$ such that $\ucd
W=2$ but $\ugd W=3$ were exhibited.  We show here that some, but not
all, of these examples have a similar property for actions with
virtually cyclic stabilisers.

\begin{maintheorem}
  Let $(W,S)$ be a right-angled Coxeter system for which the nerve
  $L=L(W,S)$ is an acyclic $2$-complex that cannot be embedded in any
  contractible $2$-complex.
  \begin{itemize}
  \item If $W$ is word hyperbolic, then
    \begin{equation*}
      \uucd W = 2 \qquad \text{and} \qquad \uugd W = 3.
    \end{equation*}
    
  \item If $W$ is not word hyperbolic, then
    \begin{equation*}
      \uucd W = \uugd W \geq 3.
    \end{equation*}
  \end{itemize}
\end{maintheorem}

A right angled Coxeter group $W$ is word hyperbolic if and only if its
nerve $L$ satisfies the so called ``flag no squares condition'',
cf.~\cite{davis-08}*{p.~233}. By Proposition~2.1 of~\cite{dranishnikov-99} the
``flag no squares condition'' puts no restriction on the homeomorphism type of
the $2$-complex~$L$ (or see~\cite{brady-01}*{p.~498} for an explicit example
for a suitably triangulated~$L$). Therefore it follows from our theorem, that
the Bredon analogue of the Eilenberg--Ganea Conjecture is false 
for the family of virtually cyclic subgroups.

The proof of the non-word hyperbolic case of our Main Theorem is the
easy part and is described in Section~\ref{sec:non-wh-case}. The word
hyperbolic case is Theorem~\ref{thrm:gd} and~\ref{thrm:cd} combined.

\medskip

As mentioned before, in the classical case $\cd_\frakF G=1$
implies $\gd_\frakF G=1$ by the work of Stallings and Swan. It follows from
Dunwoody's Accessibility Theorem~\cite{dunwoody-79}, that the same
implication is true in the case that $\frakF=\Ffin(G)$. In the light
of this one may ask, whether this implication also holds in the
case that $\frakF=\Fvc(G)$. The first author obtained in his thesis a
positive answer for countable, torsion-free, soluble
groups~\cite{fluch-phdthesis}*{p.~127}. In this class, the groups $G$
with $\uucd G=1$ are precisely the subgroups of the rational numbers
which are not finitely generated and for these groups $\uugd G=1$
holds. However, a general answer to this question is still open.

\subsection*{Acknowledgments}

The first author is grateful for the support of the CRC~701 of the DFG.


\section{Coxeter Groups and the Davis Complex}

A \emph{Coxeter matrix} is a symmetric matrix $M=(m_{st})$ indexed by
a finite set~$S$ and with entries integers or $\infty$ subject to the
conditions that for all~$s,t\in S$
\begin{enumerate}
    \item $m_{st} = 1$ if $s=t$, and
    
    \item $m_{st} \geq 2$ otherwise.
\end{enumerate}
Associated to a Coxeter matrix $M$ one has the \emph{Coxeter group}
$W$ given by the presentation
\begin{equation*}
  W = \langle S \mid (st)^{m_{st}}=1 \text{ for all $s,t\in S$ with
    $m_{st}\neq \infty$} \rangle.
\end{equation*}
The Coxeter group $W$ is \emph{right-angled} if the finite
off-diagonal entries of the Coxeter matrix are all equal to $2$.  The
elements of $S$ are called the \emph{fundamental Coxeter generators} of the
Coxeter group~$W$ and the pair $(W,S)$ is called a \emph{Coxeter
  system}.  If $T\subset S$, then $W_T$ denotes the subgroup of $W$
generated by~$T$ and these subgroups are called \emph{special}.

The \emph{nerve} $L=L(W,S)$ of a Coxeter system $(W,S)$ is the
simplicial complex with vertex set $S$ and whose simplices are the
non-empty subsets~$T\subset S$ for which the special subgroup $W_T$ is
finite.

Given a Coxeter system $(W,S)$ the Davis Complex $\Sigma = \Sigma(W,S)$ is a
contractible simplicial complex on which $W$ acts with finite stabilisers; the
action of the fundamental generators $S$ is by reflections.  This complex has
been introduced in~\cite{davis-83} and it can interpreted as the barycentric
subdivison of a cell complex where the cells are in bijective correspondence
with the cosets of finite special subgroups of~$W$. This cell complex admits in
a natural way a piecewise Euclidean metric and this metric can be shown to be
$\CAT(0)$.  The links of the $0$-cells of this complex can be identified with
the nerve $L$. The full subcomplex of $\Sigma$ whose vertices correspond to the
identity cosets of the finite special subgroups is denoted by $K$. It is a
fundamental domain of the action of $W$ and it can be realised as the cone
of~$L$, where~$L$ is identified with the boundary $\partial K$ in $\Sigma$.  For
details see~\cite{davis-08}.

If $(W,S)$ is a right angled Coxeter system, then its nerve is a flag
complex~\cite{davis-08}*{p.~125}.  Conversely, if we are given a finite flag
complex $L$, then we can construct a Coxeter system $(W,S)$ such that $L$ is its
nerve as follows: let $S$ be the set of vertices of $L$ and for $s\neq t$ set
$m_{st} = 2$ if $s$ and $t$ are adjacent in $L$ and set $m_{st} = \infty$ if no
edge connects $s$ and $t$ in $L$.


\section{The Non-Hyperbolic Case}
\label{sec:non-wh-case}

It suffices to show that $\uucd W\geq 3$. For this it is enough to show that
$W$ contains a subgroup $H$ with $\uucd H\geq 3$. Since $W$ is not
word hyperbolic it contains a subgroup isomorphic to
$\Z^2$~\cite{davis-08}*{p.~241}.

We show that $\uu E\Z^2=3$ using an explicit $3$-dimensional model $X$ for 
$\uu E\Z^2$, which was first described by Farrell.  See~\cite{farrell-jones} 
for a general construction containing this as a special case, or 
see~\cite{juan-pineda-06}
for a description of $X$ and a computation of $H_*(X/\Z^2; \Z)$ from which it
follows that $H^3(X/\Z^2; \Z)$ is a countable direct product of copies of
$\Z$.  
Theorem~4.2 in~\cite{fluch-phdthesis}*{p.~83} states that $H^{3}(X/\Z^{2})
\isom H^{3}_{\Fvc(\Z^2)}(\Z^{2}; \underline\Z)$. Hence it follows that $\uucd
\Z^2= 3$.



\section{The Geometric Dimension in the Hyperbolic Case}

Given a Coxeter system $(W,S)$ and a $W$-space $X$ we set
\begin{align*}
    X^{\singZ} & = \bigcup_{s\in S} X^{s}
    \\
    \intertext{and}
    X^{\sing} & = \{ x\in X \mid W_{x} \neq 1\}.
\end{align*}
Clearly $X^{\singZ}\subset X^{\sing}$.

\begin{lemma}
    \label{lem:K-convex}
    Let $K \subset \Sigma$ be the fundamental chamber of $\Sigma$ and 
    let $s\in S$. Then both $K$ and $K\cup sK$ are convex subsets of 
    $\Sigma$.
\end{lemma}

\begin{proof}
    For each $t\in S$ the fixedpoint set $\Sigma^{t}$ separates
    $\Sigma$ into two connected half spaces.  Denote by $H_{t}^{-}$
    the half space which does not intersect $K$ and denote by
    $\overline H_{t}^{+}$ the complement of $H^{-}_{t}$.  Then
    $\overline H^{+}_{t}$ is a convex subset of~$\Sigma$ containing~$K$. 
Then $K = \bigcap_{t\in S} \overline H_{t}^{+}$ is a convex
    subset of $\Sigma$.  Finaly,~$K\cap sK$ is convex since $K 
    \cup sK =
    K_{0}\cap sK_{0}$ where $K_{0}$ is the convex set $K_{0} =
    \bigcap_{t\in S\setminus \{s\}} \overline{H}^{+}_{t}$.
\end{proof}

\begin{lemma}
    \label{lem:aux1}
    Let $X$ be a model for $\underline EW$.  Then $X^{\singZ}$ is
    homotopy equivalent to $L$.
\end{lemma}

\begin{proof}
    Since $X$ is $W$-homotopy equivalent to $\Sigma$ it follows that
    $X^{\singZ}$ is homotopy equivalent to $\Sigma^{\singZ}$.  Thus it
    is enough that $\Sigma^{\singZ}$ is homotopy equivalent to $L$.
        
    Let $K$ be the fundamental chamber of $\Sigma$.  Then $K$ is
    complete and compact and due to Lemma~\ref{lem:K-convex} also
    convex.  Therefore, since $\Sigma$ is a $\CAT(0)$ space, there exists
    a retraction of $\Sigma$ onto $K$ which sends every point $x\in
    \Sigma\setminus K$ to the unique point $\pi(x)$ of~ $K$ which is
    nearest to $x$, cf.~\cite{bridson-99}*{p.176f.}.
    
    Let $K^{S}$ the union of all mirrors of $K$, that is 
    \begin{equation*}
	K^{S} = \{x\in K \mid x\in K\cap sK \text{ for some $s\in
	S$}\},
    \end{equation*}
    cf.~\cite{davis-08}*{p.~63, p.~127}. The set $K^{S}$ is homotopy
    equivalent to $L$~\cite{davis-08}*{p.~127}.
    
    Let $s\in S$ and $x\in \Sigma^{\singZ}\setminus K$.  Let $y = \pi(x)$. 
    Then~$sy\in sK$ and since~$K\cup sK$ is convex it follows that the midpoint
$m$   of the geodesic joining~$y$ and~$sy$ is contained in $K\cup sK$.  Since
$y$   and $sy$ have the same distance from $K\cap sK$ it follows that $m\in
K\cap   sK$. In particular $m\in K$. Since $x\in \Sigma^{\singZ}$ it follows
that   $d(x,y) = d(sx,sy) = d(x,sy)$. Since the metric of~$\Sigma$ is $\CAT(0)$
it   follows that $d(x,m) \leq \max(d(x,y), d(x,sy)) = d(x,y)$.  By the   
uniqueness of the point $\pi(x)$ it follows that $m=y$.  Hence $y\in K^{S}$.
    
    It follows that the homotopy equivalence $\Sigma\homotop K$
    restricts to a homotopy equivalence $\Sigma^{\singZ}\homotop K^{S}$. 
    Thus $X^{\singZ}\homotop L$.
\end{proof}

\begin{remark}
  The above lemma could be used to give a slightly different proof of
  the main assertion of Proposition~4 of~\cite{brady-01}*{p.~497}.
\end{remark}

\begin{lemma}
    \label{lem:aux2}
    Let $X$ be a model for $\uu EW$. If $W$ is word hyperbolic, then
    $X^{\singZ}$ is homotopy equivalent to
    \begin{equation*}
      L \, \vee \bigvee_{i\in I} S^{1}
    \end{equation*}
    where the index set $I$ consists of all maximal infinite virtually
    cyclic subgroups of $W$ which contain at least two non-commuting
    Coxeter generators.
\end{lemma}

\begin{proof}
  Let $Y$ be the model for $\uu EW$ which is obtained from $\Sigma$ as
  described in~\cite{juan-pineda-06}.  This construction yields for
  every maximal infinite virtually cyclic subgroup $H$ of $W$ a
  $1$-dimensional model $Z_{H}$ for $\underline EH$ together with an
  $H$-equivariant embedding $f_{H}\: Z_{H}\to \Sigma$.  We identify
  $Z_{H}$ with its image in $\Sigma$ under this embedding.  Then $Y$
  is obtained by coning off the sets $Z_{H}$ and extending the
  $W$-action suitably.

  Since $X$ is $W$-homotopy equivalent to $Y$ it follows that
  $X^{\singZ}$ is homotopy equivalent to $Y^{\singZ}$.  The set
  $Y^{\singZ}$ is obtained from $\Sigma^{\singZ}$ by coning off the
  intersection $\Sigma^{\singZ}\cap Z_{H}$ for every maximal infinite
  virtually cyclic subgroup~$H$ of~$W$.
    
  Let $s,t\in S$ such that $s,t\in H$ for some maximal infinite
  virtually cyclic subgroup~$H$ of $W$.  Then $x\in Z_{H}$ can be a
  common fixed point of~$s$ and~$t$ if and only if~$s$ and~$t$ commute.
  In particular $Z_{H}\cap X^{\singZ}$ can consist of at most~$2$
  points as a virtually cyclic subgroup of $W$ cannot contain more
  than~$2$ pairwise non-commuting Coxeter generators.  Coning off a
  singleton set of a path connected space does not change its homotopy
  type.  And coning off a subset of a path connected space which has
  two points is homotopy equivalent to attaching a $S^{1}$ to it.
  Hence the claim of the lemma follows.
\end{proof}

\begin{lemma}
    \label{lem:top}
    Let $(X,A)$ be a CW-pair and let $B$ be a CW-complex
    which is homotopy equivalent to $A$. Then there exists a
    CW-pair~$(Y,B)$ which is homotopy equivalent to $(X,A)$ such that the
cells of $X\setminus A$
    are dimension wise in a $1$-to-$1$ correspondence to the cells of
    $Y\setminus B$.
\end{lemma}

\begin{proof}
  This follows directly from Theorem~4.1.7
  in~\cite{geoghegan-08}*{p.~104}.
\end{proof}

\begin{theorem}
  \label{thrm:gd}
  Let $(W,S)$ be a Coxeter system with $W$ word hyperbolic and such
  that the nerve $L(W,S)$ of this Coxeter system is an acyclic
  complex, which is not homotopy equivalent to a subcomplex of a
  contractible $2$-complex.  Then $\uugd W=3$.
\end{theorem}

\begin{proof}
  Assume towards a contradiction that there exists a $2$-dimensional
  model $X$ for $\uu EW$.  Then $X^{\singZ}$ is homotopy equivalent to
  $L\vee \bigvee S^{1}$ by Lemma~\ref{lem:aux2}.  By
  Lemma~\ref{lem:top} there exists a $2$-dimensional CW-complex $Y$
  which is homotopy equivalent to $X$ and which contains $L\vee
  \bigvee S^{1}$.  In particular $L$ is a subcomplex of $Y$
  contradicting the assumption that $L$ does not embed into a
  contractible $2$-complex.  Thus $\uugd W\geq 3$.
    
  On the other hand, the Davis complex $\Sigma$ is a model for
  $\underline EW$ and $\dim \Sigma = \dim L + 1= 3$. Since $W$ is word
  hyperbolic we can elevate $\Sigma$ to a model for $\uu EW$ by
  attaching orbits of cells in dimension $2$ and
less,~cf.~\cite{juan-pineda-06}. Thus~$\uugd W\leq 3$ and equality holds.
\end{proof}


\section{The Cohomological Dimension}

\begin{theorem}
  \label{thrm:cd}
  Let $(W,S)$ be a Coxeter system with $W$ word hyperbolic and such
  that the nerve $L(W,S)$ of this Coxeter system is an acyclic complex
  which is not homotopy equivalent to a subcomplex of a
contractible $2$-complex. Then $\uucd W= 2$.
\end{theorem}

\begin{proof}
    Let $\frakF$ be the family of virtually cyclic subgroups of $W$.
    Let $Z$ be the submodule of the trivial $\OFW$-module
    given by $Z(G/H)=\Z$ for any finite subgroup $H$ of $W$ and which
    is $0$ otherwise.  The complex $\Sigma^{\sing}$ is acyclic and
    $2$-dimensional by~\cite{brady-01} and it follows that $\underline
    C_{*}(\Sigma^{\sing})$ gives a projective resolution of $Z$ of
    length $2$.  Thus $\pd Z\leq 2$.
    
    On the other hand, if $X$ is a model for $\underline EW$, then a
    model $Y$ for $\uu EW$ can be obtained from $X$ by attaching
    orbits of cells in dimension $2$ and
less~\cite{juan-pineda-06}*{Proposition~9}.
    It follows that $\underline C_{*}(Y,X)$ gives a free resolution of~$Q =
\underline \Z/Z$ of length $2$.  Thus $\pd Q \leq 2$.
    
    Consider the short exact sequence
    \begin{equation*}
        0 \to Z \to \underline \Z \to 
	Q \to 0
    \end{equation*}
    of $\OFW$-modules. Since $\pd Z$ and $\pd Q$ 
    are bounded by $2$ it follows by the Horseshoe Lemma that $\uupd 
    \underline \Z\leq 2$, that is~$\uucd W\leq 2$.
    
    On the other hand, it follows from~\cite{juan-pineda-06}*{Corollary~16} that
the quotient space $\uu EW/W$ has non-trivial cohomology in dimension $2$, and
thus $H_{\frakF}^{2}(W;\underline \Z)$ must be non-trivial too, cf.~Theorem~4.2
in~\cite{fluch-phdthesis}*{p.~83}.  As a consequence we get~$\uucd W\geq 2$ and
therefore the claim follows.
\end{proof}




\section*{References}

\begin{biblist}
    \bib{brady-01}{article}{
    author={Brady, N.},
    author={Leary, I.~J.},
    author={Nucinkis, B.~E.~A.},
    title={On algebraic and geometric dimensions for groups with torsion},
    date={2001},
    ISSN={0024-6107},
    journal={J. London Math. Soc. (2)},
    volume={64},
    number={2},
    pages={489\ndash 500},
    review={\MR{1853466 (2002h:57007)}},
    }

    \bib{bridson-99}{book}{
    author={Bridson, M.~R.},
    author={Haefliger, A.},
    title={Metric spaces of non-positive curvature},
    series={Grundlehren der Mathematischen Wissenschaften [Fundamental
    Principles of Mathematical Sciences]},
    publisher={Springer-Verlag},
    address={Berlin},
    date={1999},
    volume={319},
    ISBN={3-540-64324-9},
    review={\MR{1744486 (2000k:53038)}},
    }
    
    \bib{davis-83}{article}{
    author={Davis, M.~W.},
    title={Groups generated by reflections and aspherical manifolds not
    covered by {E}uclidean space},
    date={1983},
    ISSN={0003-486X},
    journal={Ann. of Math. (2)},
    volume={117},
    number={2},
    pages={293\ndash 324},
    url={http://dx.doi.org/10.2307/2007079},
    review={\MR{690848 (86d:57025)}},
    }
    
    \bib{davis-08}{book}{
    author={Davis, M.~W.},
    title={The geometry and topology of {C}oxeter groups},
    series={London Mathematical Society Monographs Series},
    publisher={Princeton University Press},
    address={Princeton, NJ},
    date={2008},
    volume={32},
    ISBN={978-0-691-13138-2; 0-691-13138-4},
    review={\MR{2360474 (2008k:20091)}},
    }
\bib{dranishnikov-99}{article}{
   author={Dranishnikov, A.~N.},
   title={Boundaries of Coxeter groups and simplicial complexes with given
   links},
   journal={J. Pure Appl. Algebra},
   volume={137},
   date={1999},
   number={2},
   pages={139--151},
   issn={0022-4049},
   review={\MR{1684267 (2000d:20069)}},
   doi={10.1016/S0022-4049(97)00202-8},
}
	
    \bib{dunwoody-79}{article}{
      author={Dunwoody, M.~J.},
      title={Accessibility and groups of cohomological dimension one},
      date={1979},
      ISSN={0024-6115},
      journal={Proc. London Math. Soc. (3)},
      volume={38},
      number={2},
      pages={193\ndash 215},
      review={\MR{531159 (80i:20024)}},
    }

    \bib{eilenberg-57}{article}{
    author={Eilenberg, S.},
    author={Ganea, T.},
    title={On the {L}usternik--{S}chnirelmann category of abstract groups},
    date={1957},
    ISSN={0003-486X},
    journal={Ann. of Math. (2)},
    volume={65},
    pages={517\ndash 518},
    review={\MR{MR0085510 (19,52d)}},
    }

   \bib{farrell-jones}{article}{
   author={Farrell, F. T.},
   author={Jones, L. E.},
   title={Computations of stable pseudoisotopy spaces for aspherical
   manifolds},
   conference={
      title={Algebraic topology Pozna\'n 1989},
   },
   book={
      series={Lecture Notes in Math.},
      volume={1474},
      publisher={Springer},
      place={Berlin},
   },
   date={1991},
   pages={59--74},
   review={\MR{1133892 (92k:57038)}},
}

    \bib{fluch-phdthesis}{article}{
    author={Fluch, M.},
    title={On {B}redon (Co-)Homological Dimensions of Groups},
    date={2010},
    journal={Ph.D.~thesis, University of Southampton},
    url={http://arxiv.org/abs/1009.4633v1},
    eprint={arXiv:1009.4633v1},
    }
           
    \bib{geoghegan-08}{book}{
    author={Geoghegan, R.},
    title={Topological methods in group theory},
    series={Graduate Texts in Mathematics},
    publisher={Springer},
    address={New York},
    date={2008},
    volume={243},
    ISBN={978-0-387-74611-1},
    review={\MR{2365352}},
    }

    \bib{juan-pineda-06}{incollection}{
    author={Juan--Pineda, D.},
    author={Leary, I.~J.},
    title={On classifying spaces for the family of virtually cyclic
    subgroups},
    date={2006},
    booktitle={Recent developments in algebraic topology},
    series={Contemp. Math.},
    volume={407},
    publisher={Amer. Math. Soc.},
    address={Providence, RI},
    pages={135\ndash 145},
    review={\MR{2248975 (2007d:19001)}},
    }
    
    \bib{luck-89}{book}{
    author={L{\"u}ck, W.},
    title={Transformation groups and algebraic {$K$}-theory},
    series={Lecture Notes in Mathematics},
    publisher={Springer-Verlag},
    address={Berlin},
    date={1989},
    volume={1408},
    ISBN={3-540-51846-0},
    note={Mathematica Gottingensis},
    review={\MR{1027600 (91g:57036)}},
    }
    
    \bib{luck-00}{incollection}{
    author={L{\"u}ck, W.},
    author={Meintrup, D.},
    title={On the universal space for group actions with compact isotropy},
    date={2000},
    booktitle={Geometry and topology: Aarhus (1998)},
    series={Contemp. Math.},
    volume={258},
    publisher={Amer. Math. Soc.},
    address={Providence, RI},
    pages={293\ndash 305},
    review={\MR{1778113 (2001e:55023)}},
    }

    \bib{stallings-68a}{article}{
      author={Stallings, J.~R.},
      title={Groups of dimension 1 are locally free},
      date={1968},
      ISSN={0002-9904},
      journal={Bull. Amer. Math. Soc.},
      volume={74},
      pages={361\ndash 364},
      review={\MR{0223439 (36 \#6487)}},
    }

    \bib{swan-69}{article}{
      author={Swan, R.~G.},
      title={Groups of cohomological dimension one},
      date={1969},
      ISSN={0021-8693},
      journal={J. Algebra},
      volume={12},
      pages={585\ndash 610},
      review={\MR{0240177 (39 \#1531)}},
    }

    \bib{symonds-05}{article}{
      author={Symonds, P.},
      title={The {B}redon cohomology of subgroup complexes},
      date={2005},
      ISSN={0022-4049},
      journal={J. Pure Appl. Algebra},
      volume={199},
      number={1--3},
      pages={261\ndash 298},
      review={\MR{2134305 (2006e:20093)}},
    }
\end{biblist}
\end{document}